\documentclass{amsart}
\usepackage{amssymb, amsmath, mathrsfs, verbatim, url}
\usepackage{marvosym}

\theoremstyle{plain}
\newtheorem{theorem}{Theorem}
\newtheorem{lemma}[theorem]{Lemma}
\newtheorem{proposition}[theorem]{Proposition}
\newtheorem{corollary}[theorem]{Corollary}

\theoremstyle{definition}

\newtheorem{question}{Question}

\newcommand{\cl}{\mathsf{cl}}
\newcommand{\dom}{\mathsf{dom}}
\newcommand{\ran}{\mathsf{ran}}

\newcommand{\im}{\mathfrak{i}(\mathsf{meager})}
\newcommand{\nonm}{\mathsf{non(meager)}}
\newcommand{\addm}{\mathsf{add(meager)}}
\newcommand{\covm}{\mathsf{cov(meager)}}
\newcommand{\cofm}{\mathsf{cof(meager)}}
\newcommand{\Fs}{\mathsf{F_\sigma}}
\newcommand{\Gd}{\mathsf{G_\delta}}
\newcommand{\MAc}{\mathsf{MA(countable)}}
\newcommand{\ZFC}{\mathsf{ZFC}}
\newcommand{\CH}{\mathsf{CH}}

\newcommand{\Aa}{\mathcal{A}}

\newcommand{\DD}{\mathcal{D}}

\newcommand{\FF}{\mathcal{F}}
\newcommand{\GG}{\mathcal{G}}

\newcommand{\KK}{\mathcal{K}}
\newcommand{\XX}{\mathcal{X}}

\newcommand{\MM}{\mathcal{M}}

\newcommand{\RR}{\mathcal{R}}
\newcommand{\Ss}{\mathcal{S}}
\newcommand{\PPP}{\mathbb{P}}
\newcommand{\ZZZ}{\mathbb{Z}}
\newcommand{\QQQ}{\mathbb{Q}}
\newcommand{\cccc}{\mathfrak{c}}
\newcommand{\dddd}{\mathfrak{d}}

\begin{document}

\title{Non-meager free sets and independent families}

\author{Andrea Medini}
\address{Kurt G\"odel Research Center for Mathematical Logic
\newline\indent University of Vienna
\newline\indent W\"ahringer Stra{\ss}e 25
\newline\indent A-1090 Wien, Austria}
\email{andrea.medini@univie.ac.at}
\urladdr{http://www.logic.univie.ac.at/\~{}medinia2/}

\author{Du\v{s}an Repov\v{s}}
\address{Faculty of Education, and Faculty of Mathematics and Physics
\newline\indent University of Ljubljana
\newline\indent Kardeljeva Plo\v{s}\v{c}ad 16
\newline\indent Ljubljana, 1000, Slovenia}
\email{dusan.repovs@guest.arnes.si}
\urladdr{http://www.pef.uni-lj.si/repovs/}

\author{Lyubomyr Zdomskyy}
\address{Kurt G\"odel Research Center for Mathematical Logic
\newline\indent University of Vienna
\newline\indent W\"ahringer Stra{\ss}e 25
\newline\indent A-1090 Wien, Austria}
\email{lyubomyr.zdomskyy@univie.ac.at}
\urladdr{http://www.logic.univie.ac.at/\~{}lzdomsky/}


\thanks{The first-listed author acknowledges the support of the FWF grant M 1851-N35. The second-listed author acknowledges the support of the SRA grant P1-0292-0101. The third-listed author acknowledges the support of the FWF grant I 1209-N25. The third-listed author also thanks the Austrian Academy of Sciences for its generous support through the APART Program.}

\date{June 20, 2017}

\begin{abstract}
Our main result is that, given a collection $\RR$ of meager relations on a Polish space $X$ such that $|\RR|\leq\omega$, there exists a dense Baire subspace $F$ of $X$ (equivalently, a nowhere meager subset $F$ of $X$) such that $F$ is $R$-free for every $R\in\RR$. This generalizes a recent result of Banakh and Zdomskyy. As an application, we show that there exists a non-meager independent family on $\omega$, and define the corresponding cardinal invariant. Furthermore, assuming Martin's Axiom for countable posets, our result can be strengthened by substituting ``$|\RR|\leq\omega$'' with ``$|\RR|<\cccc$'' and ``Baire'' with ``completely Baire''.
\end{abstract}

\maketitle

\section{Introduction}

Given a set $X$, we say that $R$ is a \emph{relation} on $X$ if $R\subseteq X^n$ for some $n=n_R$ such that $1\leq n<\omega$. By \emph{space} we mean separable metrizable topological space. A space is \emph{crowded} if it is non-empty and has no isolated points. A subset $S$ of a space $X$ is \emph{meager} if there exist closed nowhere dense subsets $C_k$ of $X$ for $k\in\omega$ such that $S\subseteq\bigcup_{k\in\omega}C_k$. A subset $S$ of a space $X$ is \emph{comeager} if $X\setminus S$ is meager. A space $X$ is \emph{Baire} if every non-empty open subset of $X$ is non-meager in $X$. A subset $S$ of a space $X$ is \emph{nowhere meager} if $S\cap U$ is non-meager in $X$ for every non-empty open subset $U$ of $X$. We will be freely using the following easy proposition (see \cite[Exercise A.13.7]{vanmill}).
\begin{proposition}
Let $X$ be a space. For a subset $S$ of $X$, the following conditions are equivalent:
\begin{itemize}
\item $S$ is nowhere meager in $X$.
\item $S$ is dense in $X$ and Baire as a subspace of $X$.
\end{itemize}
\end{proposition}

A relation $R$ on a space $X$ is \emph{meager} if $R$ is a meager subset of $X^n$, where $n=n_R$. Given a relation $R$ on a set $X$, we say that $F\subseteq X$ is \emph{$R$-free} if $x\notin R$ whenever $x:n_R\longrightarrow F$ is injective. Given a collection $\RR$ consisting of relations on a set $X$, we say that $F\subseteq X$ is \emph{$\RR$-free} if $F$ is $R$-free for every $R\in\RR$.

The following result (see \cite[Section 6]{kuratowski} or \cite[Exercise 8.8 and Theorem 19.1]{kechris}) has become a standard tool in mathematics. It guarantees the existence of nice free sets for a small number of small relations.
\begin{theorem}[Kuratowski]\label{perfect}
Let $\RR$ be a collection of meager relations on a crowded Polish space $X$ such that $|\RR|\leq\omega$.\footnote{\,In fact, it is clear from the proof that ``$|\RR|\leq\omega$'' can be weakened to ``$|\RR|<\covm$''. Furthermore, the example at the beginning of Section 8 shows that the bound $\covm$ is optimal.} Then there exists $F\subseteq X$ that satisfies the following conditions.
\begin{itemize}
\item $F$ is homeomorphic to $2^\omega$.
\item $F$ is $\RR$-free.
\end{itemize}
\end{theorem}
Our main results (Theorems \ref{mainpolish} and \ref{maincbpolish}) are in the same vein, except that by ``nice'' we will mean respectively ``nowhere meager'' and ``dense and completely Baire'' instead of ``homeomorphic to $2^\omega$''. Theorem \ref{mainpolish} generalizes (and was inspired by) a recent result of Banakh and Zdomskyy, which concerns the case $\RR=\{R\}$, where $R$ is a binary relation (see \cite[Theorem 1]{banakhzdomskyy}). First, we will give proofs for crowded Polish spaces (see Theorems \ref{main} and \ref{maincb}). As an application, we will show that there exists a non-meager independent family in $\ZFC$ (see Theorem \ref{nonmeagerindependent}).

\section{More notation and terminology}

Assume that a countably infinite set $I$ is given (in most cases $I=\omega$). Throughout this paper, we freely identify subsets of $I$ with their characteristic functions. Accordingly, we say that a collection of subsets of $I$ is \emph{meager} if it is meager as a subset of $2^I$. Given a collection $\XX\subseteq 2^I$ and $A\subseteq I$, define
$$
\XX\upharpoonright A=\{x\upharpoonright A:x\in\XX\}.
$$

A \emph{filter} on $I$ is a collection of non-empty subsets of $I$ that is closed under finite intersections and supersets. Furthermore, we assume that $\{x\subseteq I:|I\setminus x|<\omega\}\subseteq\FF$ for every filter $\FF$ on $I$. If the set $I$ is not mentioned, we will assume that $I=\omega$. We will freely use the fact that a filter on $I$ is non-meager if and only if it is Baire as a subspace of $2^I$ (see for example \cite[Section 2]{medinimilovich}).

Given $x\subseteq I$, define $x^0=I\setminus x$ and $x^1=x$. An \emph{independent family} on $I$ is a collection $\Aa$ consisting of subsets of $I$ such that $\bigcap_{x\in F}x^{\nu(x)}$ is infinite for every non-empty $F\in [\Aa]^{<\omega}$ and $\nu:F\longrightarrow 2$. Once again, if the set $I$ is not mentioned, we will assume that $I=\omega$. Notice that an independent family might be Baire as a subspace of $2^\omega$ without being non-meager. In fact, it is well-known that Theorem \ref{perfect} implies the existence of independent families that are homeomorphic to $2^\omega$ (see the proof of Theorem \ref{nonmeagerindependent}), and these are necessarily closed nowhere dense in $2^\omega$.

A space $X$ is \emph{completely Baire}\footnote{\,Some authors use ``hereditarily Baire'' or even ``hereditary Baire'' instead of ``completely Baire''.} if every closed subspace of $X$ is Baire. The following classical result (see \cite[Corollary 21.21]{kechris} and \cite[Corollary 1.9.13]{vanmill}) collects the most important facts about completely Baire spaces. See also Theorem \ref{newcb}.
\begin{theorem}[Hurewicz]\label{hurewicz}
Let $X$ be a space. Consider the following conditions:
\begin{enumerate}
\item\label{polishcb} $X$ is Polish.
\item\label{cb} $X$ is completely Baire.
\item\label{noQ} $X$ does not contain a closed copy of $\QQQ$.
\end{enumerate}
The implications $(\ref{polishcb})\rightarrow(\ref{cb})\leftrightarrow (\ref{noQ})$ hold for every $X$. If $X$ is a coanalytic subspace of some Polish space then the implication $(\ref{polishcb}) \leftarrow (\ref{cb})$ holds as well.
\end{theorem}
\newpage
Recall the following definitions:
\begin{itemize}
\item $\addm$ is the minimal size of a collection $\MM$ consisting of meager subsets of $2^\omega$ such that $\bigcup\MM$ is non-meager.
\item $\cofm$ is the minimal size of a collection $\MM$ consisting of meager subsets of $2^\omega$ such that for every meager subset $N$ of $2^\omega$ there exists $M\in\MM$ such that $N\subseteq M$.
\item $\covm$ is the minimal size of a collection $\MM$ consisting of meager subsets of $2^\omega$ such that $\bigcup\MM=2^\omega$.
\item $\nonm$ is the minimal size of a non-meager subset of $2^\omega$.
\item $\dddd$ is the minimal size of a family $\FF\subseteq\omega^\omega$ such that for every $f\in\omega^\omega$ there exists $g\in\FF$ such that $f(n)\leq g(n)$ for all but finitely many values of $n$.
\end{itemize}
Since every crowded Polish space has a dense subspace homeomorphic to $\omega^\omega$ (see the first paragraph of the proof of \cite[Theorem 5.4]{medinizdomskyy} and \cite[Theorem 1.9.8]{vanmill}), it is easy to see that $2^\omega$ could have been substituted with any other crowded Polish space in the above definitions. Notice that $\dddd$ is the minimal size of a family $\KK$ consisting of compact subsets of $\omega^\omega$ such that $\bigcup\KK=\omega^\omega$. The inequalities $\addm\leq\nonm\leq\cofm$ and $\addm\leq\covm\leq\dddd\leq\cofm$ are well-known (see \cite[Section 5]{blass}). We denote by $\MAc$ the statement that Martin's Axiom holds for countable posets, which is equivalent to $\covm=\cccc$ (see \cite[Theorem 7.13]{blass}).

\section{Preliminaries on non-meager filters}

In this section, we collect all the preliminaries on non-meager filters that will be needed in the next section. All these results are well-known.

Recall that a function $\phi:I\longrightarrow J$ is \emph{finite-to-one} if $\phi^{-1}(j)$ is finite for every $j\in J$. Notice that every finite-to-one function $\phi:I\longrightarrow\omega$ induces a partition of $I$ into finite sets, namely $\{\phi^{-1}(j):j\in\omega\}\setminus\{\varnothing\}$. Conversely, given a partition $\{I_j:j\in\omega\}$ of $I$ into finite sets, setting $\phi(i)=j$ for every $i\in I_j$ yields a finite-to-one function $\phi:I\longrightarrow\omega$. Given a countably infinite set $I$, a finite-to-one $\phi:I\longrightarrow\omega$ and $x,y\in 2^I$, we will use the notation
$$
[[x=y]]=\{j\in\omega:x\upharpoonright\phi^{-1}(j)=y\upharpoonright\phi^{-1}(j)\}.
$$
The above set obviously depends on $\phi$, but what $\phi$ is will always be clear from the context. The following two results are immediate consequences of \cite[Theorem 5.2]{blass} and \cite[Proposition 9.4]{blass}. Corollary \ref{charmeagerfilter} originally appeared, with a slightly different formulation, as part of \cite[Th\'eor\`eme 21]{talagrand}.
\begin{theorem}\label{charmeager}
Let $I$ be a countably infinite set. For a subset $S$ of $2^I$, the following conditions are equivalent:
\begin{itemize}
\item $S$ is meager.
\item There exist a finite-to-one $\phi:I\longrightarrow\omega$ and $z\in 2^I$ such that $[[x=z]]$ is finite for every $x\in S$.
\end{itemize}
\end{theorem}
\begin{corollary}[Talagrand]\label{charmeagerfilter}
Let $I$ be a countably infinite set. For a filter $\FF$ on $I$, the following conditions are equivalent:
\begin{itemize}
\item $\FF$ is meager.
\item There exists a finite-to-one $\phi:I\longrightarrow\omega$ such that $[[x=z]]$ is finite for every $x\in\FF$, where $z\in 2^I$ is defined by $z(i)=0$ for every $i\in I$.
\end{itemize}
\end{corollary}

Let $I$ and $J$ be countably infinite sets. Given a finite-to-one $\phi:I\longrightarrow J$ and a filter $\FF$ on $J$, define
$$
\phi^{-1}(\FF)=\{A\subseteq I:\phi^{-1}[B]\subseteq A\textrm{ for some }B\in\FF\}.
$$
The following three lemmas are simple applications of Corollary \ref{charmeagerfilter}, and their proofs are left to the reader.

\begin{lemma}\label{restrictfilter}
Let $\FF$ be a non-meager filter, and fix $A\in\FF$. Then $\FF\upharpoonright A$ is a non-meager filter on $A$.
\end{lemma}

\begin{lemma}\label{productfilter}
Let $\FF_\ell$ be a non-meager filter for $\ell\in\omega$. Then $\bigcap_{\ell\in\omega}\FF_\ell$ is a non-meager filter.
\end{lemma}

\begin{lemma}\label{inversefilter}
Let $I$ and $J$ be countably infinite sets. Fix a finite-to-one function $\psi:I\longrightarrow J$. Let $\FF$ be a non-meager filter on $J$. Then $\psi^{-1}(\FF)$ is a non-meager filter on $I$.
\end{lemma}

\section{Nowhere meager free sets}

This section contains our main result, which is Theorem \ref{main}. In fact, as Theorem \ref{mainpolish} shows, the assumption that $X$ is crowded can be dropped. Lemma \ref{maintech} is the combinatorial core of this result, and its proof is postponed to the end of the section.

\begin{lemma}\label{maintech}
Let $\RR$ be a collection of meager relations on $2^\omega$ such that $|\RR|\leq\omega$. Then there exist nowhere meager subsets $E_\alpha$ of $2^\omega$ for $\alpha\in\cccc$ such that $x\notin R$ whenever $R\in\RR$ and $x\in\prod_{k\in n_R}E_{\alpha_k}$ for distinct $\alpha_0,\ldots,\alpha_{n_R-1}\in\cccc$.
\end{lemma}

\begin{theorem}\label{main}
Let $\RR$ be a collection of meager relations on a crowded Polish space $X$ such that $|\RR|\leq\omega$. Then there exists $F\subseteq X$ that satisfies the following conditions:
\begin{itemize}
\item $F$ is dense in $X$.
\item $F$ is Baire.
\item $F$ is $\RR$-free.
\end{itemize}
Furthermore, given any cardinal $\kappa$ such that $\cofm\leq\kappa\leq\cccc$, it is possible to choose $F$ so that the additional requirement $|F|=\kappa$ will be satisfied.
\end{theorem}
\begin{proof}
Since $X$ is a crowded Polish space, it contains a dense subspace $B$ that is homeomorphic to $\omega^\omega$ (see the first paragraph of the proof of \cite[Theorem 5.4]{medinizdomskyy} and \cite[Theorem 1.9.8]{vanmill}). Identify $B$ with the subspace of $2^\omega$ consisting of the sequences that are not eventually constant. Given $R\in\RR$ with $n=n_R$, let $R'=R\cap (B^n)$, and view each $R'$ as a meager relation on $2^\omega$. Let $\RR'=\{R':R\in\RR\}$.

Since $\kappa\geq\cofm$, it is possible to fix $(M_\alpha,U_\alpha)$ for $\alpha\in\kappa$ so that the following conditions will be satisfied:
\begin{itemize}
\item Each $M_\alpha$ is a meager subset of $X$.
\item Each $U_\alpha$ is a non-empty open subset of $X$.
\item Given a meager subset $M$ of $X$ and a non-empty open subset $U$ of $X$, there exists $\alpha\in\kappa$ such that $M\subseteq M_\alpha$ and $U_\alpha\subseteq U$.
\end{itemize}
Assume without loss of generality that $\Delta\in\RR'$, where $\Delta=\{(z,z):z\in 2^\omega\}$. Let $E_\alpha$ for $\alpha\in\cccc$ be obtained by applying Lemma \ref{maintech} with $\RR=\RR'$. Notice that $\Delta\in\RR'$ will ensure that $E_\alpha\cap E_\beta=\varnothing$ whenever $\alpha\neq\beta$. Without loss of generality, assume that each $E_\alpha\subseteq B$, and notice that each $E_\alpha$ is nowhere meager in $X$. Therefore, it is possible to pick
$$
z_\alpha\in (E_\alpha\cap U_\alpha)\setminus M_\alpha
$$
for $\alpha\in\kappa$. Define $F=\{z_\alpha:\alpha\in\kappa\}$. It is easy to check that $F$ has the desired properties.
\end{proof}
\begin{corollary}\label{addcov}
Let $\RR$ be a collection of meager relations on a crowded Polish space $X$ such that $|\RR|<\addm$. Then there exists a nowhere meager $F\subseteq X$ that is $\RR$-free.
\end{corollary}
\begin{proof}
Define $R_n'=\bigcup\{R:R\in\RR\text{ and }n_R=n\}$ for $1\leq n<\omega$, and notice that each $R_n'$ is a meager relation on $X$ because $|\RR|<\addm$. Let $\RR'=\{R_n':1\leq n<\omega\}$. Since every $\RR'$-free subset of $X$ is clearly $\RR$-free, an application of Theorem \ref{main} will conclude the proof.
\end{proof}

\begin{proof}[Proof of Lemma \ref{maintech}]
Using Theorem \ref{charmeager}, one can easily construct a single finite-to-one function $\phi:\omega\longrightarrow\omega$ and $x_R\in (2^\omega)^{n_R}$ for $R\in\RR$ so that
$$
R\subseteq\left\{x\in (2^\omega)^{n_R}:\bigcap_{k\in n_R}[[x(k)=x_R(k)]]\textrm{ is finite}\right\}
$$
for every $R\in\RR$. Let $\Ss=\{(k,R):k\in n_R\}$, and assume without loss of generality that $\Ss$ is infinite.

Since $\Ss$ is a countable set, we can apply \cite[Exercise III.2.12]{kunen} and fix a family $\GG\subseteq \Ss^\omega$ such that $|\GG|=\cccc$ and
$$
\bigcap_{k\in n}g_k^{-1}(s_k)\textrm{ is infinite}
$$
whenever $1\leq n<\omega$, $s_0,\ldots, s_{n-1}\in\Ss$, and $g_0,\ldots, g_{n-1}$ are distinct elements of $\GG$. Write $\omega=\bigcup_{\ell\in\omega}\Omega_\ell$, where the sets $\Omega_\ell$ are infinite and pairwise disjoint, and fix a non-meager filter $\FF_\ell$ for $\ell\in\omega$ such that $\Omega_\ell\in\FF_\ell$.

We claim that the sets
$$
E_g=\bigcap_{(k,R)\in\Ss}\left\{z\in 2^\omega:[[z=x_R(k)]]\in\bigcap_{\ell\in g^{-1}(k,R)}\FF_\ell\right\}
$$
for $g\in\GG$ have the desired properties. Fix $R\in\RR$, distinct $g_0,\ldots, g_{n_R-1}\in\GG$ and $x\in\prod_{k\in n_R}E_{g_k}$. By the choice of $\GG$, there exists $\ell\in\omega$ such that $\ell\in g_k^{-1}(k,R)$ for every $k\in n_R$. In particular, $[[x(k)=x_R(k)]]\in\FF_\ell$ for every $k\in n_R$. Therefore $\bigcap_{k\in n_R}[[x(k)=x_R(k)]]\in\FF_\ell$, which implies $x\notin R$.

It remains to show that each $E_g$ is dense in $2^\omega$ and Baire. Fix $g\in\GG$, and let $E=E_g$. Define $J_s=\bigcup_{\ell\in g^{-1}(s)}\Omega_\ell$ and $I_s=\phi^{-1}[J_s]$ for $s\in\Ss$. Since $E$ is closed under finite modifications of its elements, in order to prove that $E$ is dense in $2^\omega$, it will be enough to show that $E\neq\varnothing$. It will be sufficient to construct $z\in 2^\omega$ such that $[[z=x_R(k)]]\supseteq J_{(k,R)}$ for every $(k,R)\in\Ss$. This is easily achieved by setting $z\upharpoonright I_{(k,R)}=x_R(k)\upharpoonright I_{(k,R)}$ for every $(k,R)\in\Ss$.

Finally, we will show that $E$ is Baire. It is easy to check that $z\in E$ if and only if $z\upharpoonright I_s\in E\upharpoonright I_s$ for every $s\in\Ss$. It follows that $E$ is homeomorphic to $\prod_{s\in\Ss}E\upharpoonright I_s$. Since a countable product of Baire spaces is a Baire space (see \cite[Theorem 3]{oxtoby} or \cite[Exercise A.6.11]{vanmill}), it will be enough to show that each $E\upharpoonright I_s$ is Baire.

Fix $s=(k,R)\in\Ss$, and let $J=J_s$, $I=I_s$. Define $\FF=\bigcap_{\ell\in g^{-1}(s)}\FF_\ell$. Notice that $\FF$ is a non-meager filter by Lemma \ref{productfilter}. Also define $\psi=\phi\upharpoonright I:I\longrightarrow J$. By considering an appropriate homeomorphism of $2^I$, we can assume without loss of generality that $x_R(k)(i)=1$ for every $i\in I$. Under this assumption, it is easy to realize that $E\upharpoonright I=\psi^{-1}(\FF\upharpoonright J)$. It follows from Lemmas \ref{restrictfilter} and \ref{inversefilter} that $E\upharpoonright I$ is Baire.
\end{proof}

\section{Dense completely Baire free sets}

The main result of this section is Theorem \ref{maincb}, which shows that Theorem \ref{main} can be considerably strengthened under the assumption of $\MAc$. Once again, this generalizes to arbitrary Polish spaces (see Theorem \ref{maincbpolish}). We will need several preliminary lemmas.

Let $X$ be a set and $1\leq n<\omega$. Given $A\subseteq X^n$ and $x\in X^{n-1}$, we will use the notation
$$
A[x]=\{z\in X:x^{\frown}z\in A\}.
$$
Notice that if $n=1$ and $x=\varnothing$ then $A[x]=A$. The following is a special case of a classical result (see \cite[Theorem 8.41]{kechris} for a proof).
\begin{lemma}[Kuratowski, Ulam]\label{fubini}
Let $X$ be a space and $2\leq n<\omega$. If $A$ is a comeager subset of $X^n$ then there exists a comeager subset $B$ of $X^{n-1}$ such that $A[x]$ is comeager in $X$ for every $x\in B$.
\end{lemma}

Notice that every meager relation is contained in a meager $\Fs$ relation. Given a bijection $\pi:n\longrightarrow n$, define $h_\pi:X^n\longrightarrow X^n$ by setting $h_\pi(x)(k)=x(\pi(k))$ for every $x\in X^n$. We say that $R$ is \emph{symmetric} if $h_\pi[R]=R$ for every bijection $\pi:n_R\longrightarrow n_R$. Using the fact that each $h_\pi$ is a homeomorphism, it is easy to see that every meager relation is contained in a symmetric meager $\Fs$ relation.

Assume that $R$ is a symmetric meager $\Fs$ relation on a space $X$ with $n=n_R$. Using Lemma \ref{fubini}, it is easy to recursively construct subsets $G^\ell_R$ of $X^\ell$ for $1\leq\ell\leq n$ such that the following properties are satisfied:
\begin{itemize}
\item $G^n_R=X^n\setminus R$.
\item Each $G^\ell_R$ is a symmetric dense $\Gd$ subset of $X^\ell$.
\item $G^{\ell+1}_R[x]$ is a dense $\Gd$ subset of $X$ for every $x\in G^\ell_R$.
\end{itemize}

Let $\RR$ be a collection consisting of relations on a set $X$. Given $F\subseteq X$, define the following condition:
\begin{enumerate}
\item[$\circledast(F,\RR)$] If $R\in\RR$, $0\leq\ell<n_R$ and $x:\ell+1\longrightarrow F$ is injective, then $x\in 
G^{\ell+1}_R$.
\end{enumerate}
Notice that $\circledast(F,\RR)$ implies that $F$ is $\RR$-free.

\begin{lemma}\label{denseGdelta}
Let $\RR$ be a collection of symmetric meager $\Fs$ relations on a space $X$. Let $F\subseteq X$ be such that $\circledast(F,\RR)$ holds. Fix $R\in\RR$, $0\leq\ell<n_R$ and an injection $x:\ell\longrightarrow F$. Then $G^{\ell+1}_R[x]$ is a dense $\Gd$ subset of $X$.
\end{lemma}
\begin{proof}
If $\ell=0$ then $G^{\ell+1}_R[x]=G^1_R$ is a dense $\Gd$ subset of $X$ by construction. Now assume that $\ell\geq 1$. By applying condition $\circledast(F,\RR)$ to $x:(\ell-1)+1\longrightarrow F$, one sees that $x\in G^\ell_R$. Therefore $G^{\ell+1}_R[x]$ is a dense $\Gd$ subset of $X$ by construction.
\end{proof}

\begin{lemma}\label{ddenseGd}
Let $\GG$ be a non-empty collection of dense $\Gd$ subsets of a crowded Polish space $X$ such that $|\GG|<\dddd$. Assume that $\bigcap\GG$ is dense in $X$. Then $\bigcap\GG$ is non-meager in $X$.
\end{lemma}
\begin{proof}
Since every crowded Polish space has a dense $\Gd$ zero-dimensional subspace (see the first paragraph of the proof of \cite[Theorem 5.4]{medinizdomskyy}), we can assume without loss of generality that $X$ is zero-dimensional. In particular, we can assume that $X$ is a dense (necessarily $\Gd$) subspace of $2^\omega$.

Fix a countable dense subset $D$ of $\bigcap\GG$. Let $B=2^\omega\setminus D$, and observe that $X\setminus D$ is a dense $\Gd$ subset of $B$. Since $B$ is homeomorphic to $\ZZZ^\omega$ (see \cite[Theorem 1.9.8]{vanmill}), we can fix a binary operation $\cdot$ on $B$ that makes $B$ a topological group. Let $N=\bigcup\{2^\omega\setminus G:G\in\GG\}$, and notice that $N$ can be written as the union of less than $\dddd$ compact subsets of $B$. Assume, in order to get a contradiction, that $N$ is comeager in $B$. Since $B$ is homeomorphic to $\omega^\omega$, it will be enough to show that $N\cdot N^{-1}=B$. To see this, fix an arbitrary $x\in B$ and observe that $(x\cdot N)\cap N$ is comeager in $B$, hence non-empty. This means that $x\cdot y=z$ for some $y,z\in N$, hence $x=z\cdot y^{-1}\in N\cdot N^{-1}$. In conclusion, we see that $B\setminus N=\bigcap\GG\setminus D$ is non-meager in $B$. It follows that $\bigcap\GG$ is non-meager in $X$.
\end{proof}

\begin{lemma}\label{dense}
Assume that $\dddd=\cofm$. Let $\RR$ be a collection of meager relations on a crowded Polish space $X$ such that $|\RR|<\covm$. Then there exists a nowhere meager $F\subseteq X$ such that condition $\circledast(F,\RR)$ holds.
\end{lemma}
\begin{proof}
Without loss of generality, assume that $\RR$ is non-empty and each $R\in\RR$ is a symmetric meager~$\Fs$. Let $\kappa=\dddd=\cofm$. Fix a collection $\{C_\alpha:\alpha\in\kappa\}$ consisting of comeager subsets of $X$ such that for every comeager subset $C$ of $X$ there exists $\alpha\in\kappa$ such that $C_\alpha\subseteq C$. Fix a countable base $\{U_i:i\in\omega\}$ for $X$.

We will construct an increasing sequence $\langle F_\alpha:\alpha\in\kappa\rangle$ of subsets of $X$ by transfinite recursion. In the end, set $F=\bigcup_{\alpha\in\kappa}F_\alpha$. By induction, we will make sure that the following requirements are satisfied:
\begin{enumerate}
\item $|F_\alpha|<\kappa$ for every $\alpha\in\kappa$.
\item\label{nowheremeager} If $\alpha\in\kappa$ is zero or a limit ordinal and $i\in\omega$ then $U_i\cap C_\alpha\cap F_{\alpha+i+1}\neq\varnothing$.
\item\label{football1} Condition $\circledast(F_\alpha,\RR)$ holds for every $\alpha\in\kappa$.
\end{enumerate}
It is straightforward to check that condition $(\ref{nowheremeager})$ will ensure that $F$ is nowhere meager in $X$. On the other hand, condition $(\ref{football1})$ will guarantee that $\circledast(F,\RR)$ holds.

Start by letting $F_0=\varnothing$. Take unions at limit stages. At a successor stage $\beta=\alpha+i+1$, where $\alpha<\kappa$ is zero or a limit ordinal and $i\in\omega$, assume that $F_{\alpha+i}$ is given. Define
$$
\GG=\{G^{\ell+1}_R[x]\cup\ran(x):R\in\RR, 0\leq\ell<n_R, x:\ell\longrightarrow F_{\alpha+i}\textrm{ is injective}\},
$$
and notice that $\GG$ consists of dense $\Gd$ sets by Lemma \ref{denseGdelta}. We claim that it is possible to pick $z\in U_i\cap C_\alpha\cap\bigcap\GG$. If $\alpha=0$, this follows from the fact that $|\GG|<\covm$. Now assume $\alpha>0$. In this case, it is easy to check that $F_\omega\subseteq F_{\alpha+i}\subseteq\bigcap\GG$. Therefore $U_i\cap\bigcap\GG$ is non-meager in $U_i$ by Lemma \ref{ddenseGd}. Put $F_\beta=F_{\alpha+i}\cup\{z\}$. It remains to verify that condition $\circledast(F_\beta,\RR)$ holds.

Fix $R\in\RR$, $0\leq\ell<n_R$ and an injection $x:\ell+1\longrightarrow F_\beta$. We need to show that $x\in G^{\ell+1}_R$. If $z\notin\ran(x)$, this already follows from condition $\circledast(F_{\alpha+i},\RR)$, so assume that $z\in\ran(x)$. Since $G^{\ell+1}_R$ is symmetric by construction, we can assume without loss of generality that $x(\ell)=z$. Let $x'=x\upharpoonright\ell$, and notice that $z\notin\ran(x')$ because $x$ is injective. Therefore $z\in G^{\ell+1}_R[x']$, which implies $x=x'^{\frown}z\in G^{\ell+1}_R$.
\end{proof}

It is easy to realize that the assumption of $\MAc$ in the following theorem can be weakened to $\dddd=\cccc$, provided $|\RR|<\covm$.
However, since it is one of our main results, we preferred to give the following more ``quotable'' formulation. The same remark holds for Theorem \ref{maincbpolish}.

\begin{theorem}\label{maincb}
Assume that $\MAc$ holds. Let $\RR$ be a collection of meager relations on a crowded Polish space $X$ such that $|\RR|<\cccc$. Then there exists $F\subseteq X$ that satisfies the following conditions:
\begin{itemize}
\item $F$ is dense in $X$.
\item $F$ is completely Baire.
\item $F$ is $\RR$-free.
\end{itemize}
\end{theorem}
\begin{proof}
Without loss of generality, assume that each $R\in\RR$ is a symmetric meager~$\Fs$. Enumerate as $\{Q_\alpha:\alpha\in\cccc\}$ all copies of $\QQQ$ in $X$, making sure to list each one cofinally often. We will construct an increasing sequence $\langle F_\alpha:\alpha\in\cccc\rangle$ of subsets of $X$ by transfinite recursion. In the end, set
$F=\bigcup_{\alpha\in\cccc}F_\alpha$. By induction, we will make sure that the following requirements are satisfied:
\begin{enumerate}
\item\label{small} $|F_\alpha|<\cccc$ for every $\alpha\in\cccc$.
\item\label{killQ} If $Q_\alpha\subseteq F_\alpha$ for some $\alpha\in\cccc$, then $F_{\alpha+1}\cap(\cl(Q_\alpha)\setminus Q_\alpha)\neq\varnothing$.
\item\label{football2} Condition $\circledast(F_\alpha,\RR)$ holds for every $\alpha\in\cccc$.
\end{enumerate}
Using Theorem \ref{hurewicz}, it is straightforward to check that condition $(\ref{killQ})$ will ensure that $F$ is completely Baire. On the other hand, condition $(\ref{football2})$ will guarantee that condition $\circledast(F,\RR)$ holds.

Start by letting $F_0$ be a countable dense subset of the $\RR$-free set given by Lemma~\ref{dense}, thus ensuring that $F$ will be dense in $X$. Take unions at limit stages. At a successor stage $\beta=\alpha+1$, assume that $F_\alpha$ is given. First assume that $Q_\alpha\nsubseteq F_\alpha$. In this case, simply set $F_\beta=F_\alpha$. Now assume that $Q_\alpha\subseteq F_\alpha$. Apply Lemma \ref{key} with $F=F_\alpha$ and $Q=Q_\alpha$ to get $z\in\cl(Q_\alpha)\setminus Q_\alpha$ such that condition $\circledast(F_\alpha\cup\{z\},\RR)$ is satisfied. Finally, set $F_\beta=F_\alpha\cup\{z\}$.
\end{proof}

\begin{lemma}\label{key}
Let $\RR$ be a collection of symmetric meager $\Fs$ relations on a crowded Polish space $X$ such that $|\RR|<\dddd$. Assume that $F$ and $Q$ satisfy the following requirements:
\begin{itemize}
\item $|F|<\dddd$.
\item $Q\subseteq F$ is countable and crowded.
\item Condition $\circledast(F,\RR)$ holds.
\end{itemize}
Then there exists $z\in\cl(Q)\setminus Q$ such that condition $\circledast(F\cup\{z\},\RR)$ holds.
\end{lemma}
\begin{proof}
Without loss of generality, assume that $\RR$ is non-empty. Define
$$
\GG=\{(G^{\ell+1}_R[x]\cup\ran(x))\cap\cl(Q):R\in\RR, 0\leq\ell<n_R,x:\ell\longrightarrow F\textrm{ is injective}\},
$$
and observe that $|\DD|<\dddd$. By Lemma \ref{denseGdelta}, the collection $\GG$ consists of $\Gd$ subsets of the crowded Polish space $\cl(Q)$. Furthermore, it is easy to check that $Q\subseteq F\cap\cl(Q)\subseteq\bigcap\GG$. Therefore $\bigcap\GG$ is non-meager in $\cl(Q)$ by Lemma \ref{ddenseGd}, and it is possible to pick $z\in\bigcap\GG\setminus Q$.

It remains to verify that condition $\circledast(F\cup\{z\},\RR)$ holds. Fix $R\in\RR$, $0\leq\ell<n_R$ and an injection $x:\ell+1\longrightarrow F\cup\{z\}$. We need to show that $x\in G^{\ell+1}_R$. If $z\notin\ran(x)$, this already follows from condition $\circledast(F,\RR)$, so assume that $z\in\ran(x)$. Since $G^{\ell+1}_R$ is symmetric by construction, we can assume without loss of generality that $x(\ell)=z$. Let $x'=x\upharpoonright\ell$, and notice that $z\notin\ran(x')$ because $x$ is injective. Therefore $z\in G^{\ell+1}_R[x']$, which implies $x=x'^{\frown}z\in G^{\ell+1}_R$.
\end{proof}

\section{Applications to independent families}

Given a certain kind of combinatorial object on $\omega$ (such as a filter, an almost disjoint family, or an independent family),
it is natural to ask how ``big'' such an object can be (in the sense of cardinality, Baire category, or Lebesgue measure).
In keeping with the rest of the paper, by ``big'' we will mean ``big in the sense of Baire category''.
For example, it is easy to see that every maximal filter (that is, ultrafilter) is non-meager, and the existence of such filters is guaranteed by Zorn's Lemma. On the other hand, every almost disjoint family must be meager. Furthermore, it is an easy exercise to show that no filter (hence no independent family) can be comeager.

However, to the best of our knowledge, the existence of non-meager independent families in $\ZFC$ was an open problem. In analogy with the case of filters, one might wonder whether every maximal independent family is non-meager. The following proposition shows that this is not the case.

\begin{proposition}[Milovich]\label{dave}
There exists a meager maximal independent family.
\end{proposition}
\begin{proof}
Fix an infinite coinfinite $I\subseteq\omega$. Let $\Aa$ be a maximal independent family on $I$. Define $x^\ast=x\cup (\omega\setminus I)$ for $x\subseteq I$, then let $\Aa^\ast=\{x^\ast:x\in\Aa\}$. Since $\{z\subseteq\omega:(\omega\setminus I)\subseteq z\}$ is closed nowhere dense in $2^\omega$, it is clear that $\Aa^\ast$ is nowhere dense, hence meager. Furthermore, it is easy to check that $\Aa^\ast$ is an independent family. Before continuing the proof, we clarify one bit of notation. Given $x\subseteq I$, we let $x^0=I\setminus x$ and $x^1=x$. On the other hand, given $x\subseteq I$, we let $(x^\ast)^0=\omega\setminus (x^\ast)$ and $(x^\ast)^1=x^\ast$.

It remains to show that $\Aa^\ast$ is maximal. Fix $z\subseteq\omega$ such that $z\notin\Aa^\ast$. We need to show that $\Aa^\ast\cup\{z\}$ is not an independent family. Let $z'=z\cap I$. First assume that $z'\in\Aa$. Then $(z')^\ast\in\Aa^\ast$. The fact that $z\cap (\omega\setminus (z')^\ast)=\varnothing$ concludes the proof in this case. Now assume that $z'\notin\Aa$. Since $\Aa$ is maximal, by adding an element to $F$ if necessary, we can fix $F\in[\Aa]^{<\omega}$, $\delta\in 2$ and $\nu:F\longrightarrow 2$ satisfying the following conditions:
\begin{itemize}
\item There exists $x\in F$ such that $\nu(x)=0$.
\item $w\cap ((z')^\delta)$ is finite, where $w=\bigcap\{x^{\nu(x)}:x\in F\}$.
\end{itemize}
Let $w'=\bigcap\{(x^\ast)^{\nu(x)}:x\in F\}$, and observe that $w'\subseteq I$ by the first condition. It is clear that $(x^\ast)^\varepsilon\cap I=x^\varepsilon$ for every $x\subseteq I$ and $\varepsilon\in 2$. Hence $w'=w'\cap I=w$. Furthermore, one readily sees that $z^\delta\cap I=(z')^\delta$, where $z^0=\omega\setminus z$ and $z^1=z$. It follows that $w'\cap (z^\delta)=w\cap ((z')^\delta)$, which concludes the proof.
\end{proof}

A straightforward application of Theorem \ref{main} shows that big independent families do exist in $\ZFC$. In fact, the following result has been the main motivation for the research contained in this article.

\begin{theorem}\label{nonmeagerindependent}
There exists a non-meager independent family.
\end{theorem}
\begin{proof}
Given $1\leq n<\omega$ and $\nu:n\longrightarrow 2$, define
$$
R_\nu=\left\{x\in (2^\omega)^n:\bigcap_{k\in n}x(k)^{\nu(k)}\textrm{ is finite}\right\}.
$$
Let $\RR=\{R_\nu:1\leq n<\omega, \nu:n\longrightarrow 2\}$, and observe that an $\RR$-free subset of $2^\omega$ is simply an independent family. Furthermore, it is easy to check that each $R_\nu$ is a meager relation. An application of Theorem \ref{main} concludes the proof.
\end{proof}

Similarly, it is clear that the following result can be deduced from Theorem \ref{maincb} and the remark that precedes it. Theorem \ref{cbindependent} slightly improves \cite[Theorem 26]{kunenmedinizdomskyy}, where ``$\dddd=\cccc$'' is substituted by ``$\MAc$''.
\begin{theorem}\label{cbindependent}
Assume $\dddd=\cccc$. Then there exists an independent family that is dense in $2^\omega$ and completely Baire.
\end{theorem}

\section{Arbitrary Polish spaces}

In this section, we will show that the results of Sections 4 and 5 generalize to arbitrary (that is, not necessarily crowded) Polish spaces. We will use a straightforward adaptation of the method used in \cite{banakhzdomskyy}.

\begin{theorem}\label{mainpolish}
Let $\RR$ be a collection of meager relations on a Polish space $X$. Assume that $|\RR|\leq\omega$. Then there exists $F\subseteq X$ that satisfies the following conditions:
\begin{itemize}
\item $F$ is dense in $X$.
\item $F$ is Baire.
\item $F$ is $\RR$-free.
\end{itemize}
Furthermore, if $X$ is uncountable and $\kappa$ is a cardinal such that $\cofm\leq\kappa\leq\cccc$, then it is possible to choose $F$ so that the additional requirement $|F|=\kappa$ will be satisfied.
\end{theorem}
\begin{proof}
Let $E$ be the set of isolated points of $X$. If $X$ is countable, then $E$ is dense in $X$, hence we can simply set $F=E$. Now assume that $X$ is uncountable. Then it is possible to find a crowded closed subspace $Z$ of $X$ such that $C=X\setminus Z$ is countable (see \cite[Theorem 6.4]{kechris}). Notice that $E$ is a dense subset of $C$ and $Z$ is a crowded Polish space.

Given a relation $R$ on $X$ with $n=n_R$ and a function $p$ such that $\dom(p)\subsetneq n$ and $\ran(p)\subseteq E$, define
$$
R[p]=\{q\in Z^{n\setminus\dom(p)}:p\cup q\in R\}.
$$
Identify each $Z^{n\setminus\dom(p)}$ with $Z^{|n\setminus\dom(p)|}$ through the unique increasing bijection $n\setminus\dom(p)\longrightarrow |n\setminus\dom(p)|$. Notice that each $R[p]$ will be a meager relation on $Z$ whenever $R$ is meager on $X$. Define
$$
\RR'=\{R[p]:R\in\RR\text{ and }p\text{ is a function such that }\dom(p)\subsetneq n_R\text{ and }\ran(p)\subseteq E\}.
$$
Let $F'$ be the $\RR'$-free subset of $Z$ of size $\kappa$ given by Theorem \ref{main}. Set $F=E\cup F'$.

It is clear that $F$ is dense in $X$ and Baire. In order to check that $F$ is $\RR$-free, fix $R\in\RR$ with $n=n_R$ and an injective $x\in F^n$. Assume, in order to get a contradiction, that $x\in R$. Let $p=x\upharpoonright \{i\in n:x(i)\in E\}$ and $q=x\setminus p$, and notice that $\dom(p)\subsetneq n$ because $R$ is meager. Observe that $q\in (F')^{|q|}$ is injective and $q\in R[p]$. This contradicts the fact that $F'$ is $\RR'$-free.
\end{proof}

\begin{theorem}\label{maincbpolish}
Assume that $\MAc$ holds. Let $\RR$ be a collection of meager relations on a Polish space $X$. Assume that $|\RR|<\cccc$. Then there exists $F\subseteq X$ that satisfies the following conditions:
\begin{itemize}
\item $F$ is dense in $X$.
\item $F$ is completely Baire.
\item $F$ is $\RR$-free.
\end{itemize}
\end{theorem}
\begin{proof}
Proceed as in the proof of Theorem \ref{mainpolish}.
\end{proof}

We remark that, by the following theorem (see \cite[Theorem 9.9]{medinizdomskyy}), the free set given by Theorem \ref{maincbpolish} will have size $\cccc$ whenever $X$ is uncountable.
\begin{theorem}[Medini, Zdomskyy]\label{newcb}
Every uncountable completely Baire space has size $\cccc$.
\end{theorem}

\section{Open questions}

The first two questions aim at improving Theorem \ref{main} (hence Theorem \ref{mainpolish} as well). Let $\RR$ be a collection of meager subsets of $2^\omega$ such that $|\RR|=\covm$ and $\bigcup\RR=2^\omega$, viewed as a collection of unary relations on $2^\omega$. It is clear that the only $\RR$-free subset of $2^\omega$ is the empty set. However, we do not know the answer to the following question. Notice that, by Corollary \ref{addcov} and Lemma \ref{dense} respectively, any model showing that the answer to Question \ref{covsharp} is ``no'' would have to satisfy $\addm<\covm$ and $\dddd<\cofm$.

\begin{question}\label{covsharp}
Is it possible to prove in $\ZFC$ that for every collection $\RR$ consisting of meager relations on a crowded Polish space $X$ such that $|\RR|<\covm$ there exists a nowhere meager $\RR$-free subset of $X$?
\end{question}

\begin{question}\label{cofmtononm} 
Is it possible to substitute ``$\cofm$'' with ``$\nonm$'' in the statement of Theorem \ref{main}?
\end{question}
As Corollary \ref{cohennowheremeager} shows, it is consistent that $\nonm<\cofm$ and the answer to the above question is ``yes''. Proposition \ref{cohennonmeager} can be safely assumed to be folklore.
\begin{proposition}\label{cohennonmeager}
It is consistent that $\nonm<\cofm$ and every non-meager subset of $2^\omega$ has a non-meager subset of size $\omega_1$.
\end{proposition}
\begin{proof}
Assume that $\CH$ holds in $V$. We will force with the usual poset $\PPP_{\omega_2}$ for adding $\omega_2$ Cohen reals. More precisely, given an ordinal $\alpha$, we will denote by $\PPP_\alpha$ the poset of finite partial functions from $\alpha$ to $2$, ordered by reverse inclusion. It is well-known that $\nonm<\cofm$ in $V^{\PPP_{\omega_2}}$ (see \cite[Section 11.3]{blass}).

Now work in $V^{\PPP_{\omega_2}}$. Fix a non-meager subset $A$ of $2^\omega$. Using the fact that $\CH$ holds in $V$, it is straightforward to construct a strictly increasing sequence $\langle\delta_\alpha:\alpha\in\omega_1\rangle$ consisting of elements of $\omega_2$ such that, given any dense $\Gd$ subset $G$ of $2^\omega$ coded in $V^{\PPP_{\delta_\alpha}}$, there exists $x\in 2^\omega\cap  V^{\PPP_{\delta_{\alpha+1}}}$ such that $x\in A\cap G$. Let $\delta=\sup\{\delta_\alpha:\alpha\in\omega_1\}$. Let $B=A\cap V^{\PPP_\delta}$, and notice that $|B|=\omega_1$. It is clear that $B\cap G\neq\varnothing$ for every dense $\Gd$ subset $G$ of $2^\omega$ coded in $V^{\PPP_\delta}$.

Finally, since $V^{\PPP_{\omega_2}}$ is the same as $V^{\PPP_\delta\ast\PPP_{\omega_2}}$, the arguments from \cite[Section 11.3]{blass} show that $B$ is a non-meager subset of $2^\omega$ in $V^{\PPP_{\omega_2}}$.
\end{proof}
\begin{corollary}\label{cohennowheremeager}
It is consistent that $\nonm<\cofm$ and every nowhere meager subset of $2^\omega$ has a nowhere meager subset of size $\omega_1$.
\end{corollary}

Observe that the assumption of $\MAc$ in Theorem \ref{maincb} cannot be altogether dropped. To see this, consider the example at the beginning of this section in any model of $\covm<\cccc$, such as the Sacks model (see \cite[Section 11.5]{blass}). However, we do not know the answer to the following question. 
\begin{question}
Is it possible to prove in $\ZFC$ that for every collection $\RR$ consisting of meager relations on on a crowded Polish space $X$ such that $|\RR|\leq\omega$ there exists a dense completely Baire $\RR$-free subspace of $X$?
\end{question}

\newpage

Given Theorem \ref{nonmeagerindependent}, it makes sense to define the following cardinal invariant. Let
$$
\im=\min\{|\Aa|:\Aa\textrm{ is a non-meager independent family}\}.
$$
It is easy to see that $\nonm\leq\im\leq\cofm$. In fact, the first inequality is trivial, while the second one follows from Theorem \ref{main}, as in the proof of Theorem \ref{nonmeagerindependent}. Furthermore, by Corollary \ref{cohennowheremeager}, it is consistent that $\im<\cofm$. However, we do not know if the first inequality can be strict. It is clear that a positive answer to Question \ref{cofmtononm} would also give a positive answer to the following question.
\begin{question}
Is $\im=\nonm$ provable in $\ZFC$?
\end{question}

\section{Acknowledgements}

The authors thank Dave Milovich for allowing them to include Proposition \ref{dave} (private communication with the first-listed author) in this article.

\newpage

\end{document}